\documentclass{amsart}

\usepackage{graphicx}
\usepackage{comment}
\usepackage{mathrsfs}
\usepackage{amsmath}
\usepackage{amssymb}
\usepackage{amsthm}

\newtheorem{theorem}{Theorem}[section]
\newtheorem{lemma}[theorem]{Lemma}
\newtheorem{clm}[theorem]{Claim}

\theoremstyle{definition}

\newtheorem{example}[theorem]{Example}
\newtheorem*{question}{Question}

\newtheorem{remark}{Remark}

\numberwithin{equation}{section}

\begin{document}

\title[Riemann surfaces given by generalized Cantor sets]{On the length spectrums of Riemann surfaces given by generalized Cantor sets}

\author{Erina Kinjo}
\address{Department of Mechanical Engineering,  
Ehime University, 3 Bunkyo-cho, Matsuyama, Ehime 790-8577, Japan}
\email{kinjo.erina.ax@ehime-u.ac.jp}


\subjclass[2020]{Primary 30F60; Secondary 32G15}
\date{September 14, 2023, }


\keywords{Riemann surface of infinite type, Teichm\"uller space, Generalized Cantor set, length spectrum}

\begin{abstract}
For a generalized Cantor set $E(\omega)$ with respect to a sequence $\omega=\{ q_n \}_{n=1}^{\infty} \subset (0,1)$, we consider Riemann surface $X_{E(\omega)}:=\hat{\mathbb{C}} \setminus E(\omega)$ and metrics on Teichm\"uller space $T(X_{E(\omega)})$ of $X_{E(\omega)}$. If $E(\omega) = \mathcal{C}$ ( the middle one-third Cantor set), we find that on $T(X_{\mathcal{C}})$, Teichm\"uller metric $d_T$ defines the same topology as that of the length spectrum metric $d_L$.  Also, we can easily check that $d_T$ does not define the same topology as that of $d_L$ on $T(X_{E(\omega)})$ if $\sup q_n =1$. On the other hand, it is not easy to judge whether the metrics define the same topology or not if $\inf q_n =0$. In this paper, we show that the two metrics define different topologies on $T(X_{E(\omega)})$ for some $\omega=\{ q_n \}_{n=1}^{\infty}$ such that $\inf q_n =0$.
\end{abstract}

\maketitle



\section{Introduction}





For a Riemann surface  $X$, its Teichm\"uller space $T(X)$ is a set of Teichm\"uller equivalence classes $\{ [R,f] \}$, where two pairs $(R,f)$, $(S,g)$ of Riemann surfaces $R,S$ and quasiconformal mappings $f: X \to R$,  $g: X \to S$ are Teichm\"uller equivalent if there exists a conformal mapping $h$ from $R$ to $S$ such that  $g^{-1} \circ h\circ f$ is homotopic to the identity mapping, here the homotopy keeps each point of ideal boundary $\partial X$ fixed. On $T(X$), some metrics are defined. The Teichm\"uller metric $d_T$ measures how different conformal  structures of Riemann surfaces in $T(X)$ are. On the other hand, the length spectrum  metric $d_L$ measures how different hyperbolic  structures of Riemann surfaces in $T(X)$ are. More precisely, it is defined as follows: for any hyperbolic Riemann surface $X$, let $\mathscr{C} (X)$ be a set of non-trivial and non-peripheral simple closed curves in $X$, $[\alpha]$ be the geodesic freely homotopic to $\alpha \in \mathscr{C} (X)$ and  $\ell_{X} (\alpha)$ be the hyperbolic length of $\alpha \in \mathscr{C} (X)$. For any two points $[X_1,f_1],[X_2,f_2] \in T(X)$, $d_L$ is defined by
\[ d_L ([X_1 ,f_1],[X_2 , f_2]) := \log \sup_{\alpha \in \mathscr{C} (X)}  \max \left\{ \frac{\ell_{X_1}  ([f_1 (\alpha)])}{\ell_{X_2} ([f_2 (\alpha)])},  \frac{\ell_{X_2}  ([f_2 (\alpha)])}{\ell_{X_1} ([f_1 (\alpha)])} \right\}.\]
By the definition, $d_L ([X_1 ,f_1],[X_2 , f_2]) = 0$ if and only if $\ell_{X_1} ([f_1 (\alpha)]) = \ell_{X_2} ([f_2 (\alpha)])$ for any $\alpha  \in \mathscr{C} (X)$. It is known that for any hyperbolic Riemann surface $X$ and any two points $p,q \in T(X)$,
\[ d_L(p,q) \le d_T(p,q)\]
holds (cf. \cite{Sorvali} or \cite{Wolpert}). Therefore, $d_T$ and $d_L$ define the same topology on $T(X)$ if and only if for any sequence $\{ p_n\} \subset T(X)$ such that $d_L (p_n,p_0)$ converges to zero, $d_T (p_n,p_0)$ converges to zero as $n \to \infty$. Liu (\cite{Liu}) showed that the two metrics define the same topology on $T(X)$ if $X$ is a Riemann surface of finite type (i.e. a compact surface from which at most finitely many points are removed). On the other hand, Shiga (\cite{Shiga1}) considered the metrics on reduced Teichm\"uller space $T^{\sharp}(X)$ of Riemann surface $X$ of infinite type, here the reduced one is defined by a homotopy which does not necessarily keep points of ideal boundary fixed. He gave an example of the Riemann surface $X$ such that they define different topologies on $T^{\sharp}(X)$. Also, he showed that they define the same topology on $T^{\sharp}(X)$ if $X$ admits a bounded pants decomposition, that is, $X$ has a constant $M > 0$ and a pants decomposition $\bigcup_{k=1}^{\infty} P_k$ of the convex core of $X$ such that for any $k \in \mathbb{N}$, each connected component of $\partial P_k$ is either a simple closed geodesic $\alpha_k$ satisfying $0<1/M < \ell_{X} (\alpha_k) < M$ or a puncture. After that Liu-Sun-Wei (\cite{Liu-Sun-Wei}) and Kinjo (\cite{Kinjo1}, \cite{Kinjo2}, \cite{Kinjo3}) gave sufficient conditions for the two metrics to define the same topology or different ones. In particular, in Kinjo\cite{Kinjo3}, it is shown that if a Riemann surface $X$ has bounded geometry (that is, injective radii of all points of $X \setminus \{$cusp neighborhoods$\}$ are uniformly bounded from above and below), then the two metrics define the same topology.

In this paper, we consider a Riemann surface $X_{E(\omega)}$ of infinite type given by removing a generalized Cantor set $E(\omega)$ from the Riemann sphere $\hat{\mathbb{C}}$, i.e. $X_{E(\omega)} := \hat{\mathbb{C}} \setminus E(\omega)$. A generalized Cantor $E(\omega)$ set is defined as follows.

Let $\omega = \{ q_n \}_{n=1}^{\infty} \subset (0,1)$ be a sequence. Firstly, remove an open interval  with the length $q_1$ from the closed interval $I:=[0,1] \subset \mathbb{R}$ so that the remaining closed intervals $\{I_1^1, I_1^2\}$ in $I$ have the same length. Secondly, remove an open interval with the length $q_2 |I_1^1|$ (here $| \cdot |$ means the length of the interval) from each closed interval $I_1^i$ $(i=1,2)$ so that the remaining closed intervals \{$I_2^i\}_{i=1}^4$ in $I$ have the same length (Figure \ref{can}). Inductively, continue to remove an open interval with the length $q_n |I_{n-1}^1|$ from each closed interval $I_{n-1}^i$ ($i =1,2,3,...,2^{n-1}$) so that the remaining closed intervals  \{$I_n^i\}_{i \in \mathcal{I}_n}$ ($\mathcal{I}_n :=\{ 1,2,3,...,2^n \}$) in $I$ have the same length. Put $E_k:= \bigcup_{i \in \mathcal{I}_k} I_k^i$ for each $k \in \mathbb{N}$ and define $E(\omega) := \bigcap_{k=1}^{\infty}  E_k$. We call $E(\omega)$ the generalized Cantor set for $\omega = \{ q_n \}_{n=1}^{\infty}$.

\begin{figure}[h]
\centering
\includegraphics[width=7.5 cm,bb=0 0 580 178]{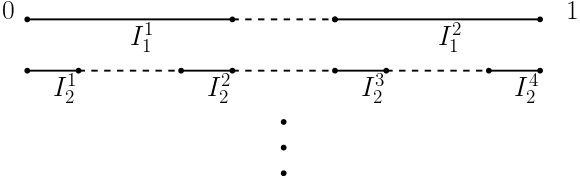}\\
\caption{$E_k = \bigcup_{i \in \mathcal{I}_k} I_k^i$ $(k=1,2).$ }
\label{can}
\end{figure}   

Now, let $\mathcal{C}$ be the middle-third Cantor set (i.e. the generalized Cantor set for $\omega = \{ q_n = \frac{1}{3} \mid n \in \mathbb{N} \}$) and put $X_\mathcal{C} := \hat{\mathbb{C}}\setminus \mathcal{C}$. Recently, Shiga (\cite{Shiga2}, \cite{Shiga3}) gives some results on Riemann surfaces given by generalized Cantor set. Theorem I in \cite{Shiga2} states that $X_\mathcal{C}$ is quasiconformally equivalent to $X_\mathcal{J} := \hat{\mathbb{C}}\setminus \mathcal{J}$ for the Julia set $\mathcal{J}$ of  some rational function. From the proof, we find that $X_\mathcal{C}$ admits a bounded pants decomposition. (We explain how to decompose $X_\mathcal{C}$ (or more precisely $X_{E(\omega)}$) in Section 2.) If $X_{E(\omega)}$ is qusiconformally equivalent to $X_{\mathcal{C}}$, then $X_{E(\omega)}$ admits a bounded pants decomposition by Wolpert's lemma (\cite{Wolpert}). Therefore, if $X_{E(\omega)}$ is qusiconformally equivalent to $X_{\mathcal{C}}$, then the Teichm\"uller metric $d_T$ and the length spectrum metric $d_L$ define the same topology on $T(X_{E(\omega)})$. Our question is whether the converse holds:

\begin{question}[1]
If the two metrics define the same topology on $T(X_{E(\omega)})$, is $X_{E(\omega)}$ quasiconformally equivalent to $X_{\mathcal{C}}$ ? (In other words,  if $X_{E(\omega)}$ is not quasiconformally equivalent to $X_{\mathcal{C}}$, do the two metrics define different topologies on $T(X_{E(\omega)})$ ?)

 \end{question}

Let us introduce new notations. For any $\omega= \{ q_n \}_{n=1}^{\infty} \subset (0,1)$, $\delta \in (0,1)$ and $i \in \mathbb{N}$, put $\omega(\delta ; i) :=\inf \{ k \in \mathbb{N} \mid q_{i+k} \ge \delta \}$, and define $\displaystyle N(\omega, \delta):=\sup_{i \in \mathbb{N}} \omega(\delta ; i)$. Note that if $N(\omega; \delta) = \infty$ for any $\delta \in (0,1)$, then $\inf q_n =0$. Indeed, if $\inf q_n > c$ for some constant $c>0$, then $\omega (\delta ,i) =1$ for any $\delta \in (0,c/2)$ and $i \in \mathbb{N}$, so  $N(\omega; \delta) =1 < \infty$. However, the converse does not hold in general. For example, for the sequence $\omega =\{q_n \}_{n=1}^{\infty}$ defined by
\[
 q_n =
 \begin{cases}
 \frac{2}{3} & (n=2m-1; m \in \mathbb{N})\\
 (\frac{1}{2})^n & (n=2m; m \in \mathbb{N}),
 \end{cases}
 \]
$\inf q_n =0$ and $N(\omega; \delta) =2 < \infty$ for any $\delta \in (0,2/3]$.

By Theorem II in \cite{Shiga3}, $X_{E(\omega)}$ is not qusiconformally equivalent to $X_{\mathcal{C}}$ if and only if $\sup q_n = 1$ or  $N(\omega; \delta) = \infty$ for any $\delta \in (0,1)$. (This is a special case of a result of MacManus (\cite{MacManus}).) Therefore, Question (1)  is rephrased as follows:

 \begin{question}[1']
 If $\sup q_n = 1$ or $N(\omega; \delta) = \infty$ for any $\delta \in (0,1)$, do the two metrics define different topologies on $T(X_{E(\omega)})$ ?
\end{question}

We find that they define different topologies on $T(X_{E(\omega)})$ if $\sup q_n = 1$. Indeed, in Shiga's paper \cite{Shiga2}, he proved that if $\sup q_n = 1$, then $X_{E(\omega)}$ is not quasiconformally equivalent to $X_{\mathcal{C}}$ by showing that under the assumption, there exists a family of simple closed geodesics $\{ \gamma_k \}_{k \in \mathbb{N}}$ such that $\ell_{X_{E(\omega)}} (\gamma_k) \to 0$ $(k \to \infty)$. On the other hand, Liu-Sun-Wei (\cite{Liu-Sun-Wei}) showed that if a  hyperbolic Riemann surface $X$ has a family of simple closed geodesics $\{ \gamma_n \}$ such that $\lim_{n \to \infty} \ell_X (\gamma_n) =0$, then the two metrics define different topologies on $T(X)$. Hence, we consider the following:

\begin{question}[2]
 If $N(\omega; \delta) = \infty$ for any $\delta \in (0,1)$, do the two metrics define different topologies on $T(X_{E(\omega)})$ ?
\end{question} 

There are two cases where $N(\omega; \delta) = \infty$ for any $\delta \in (0,1)$: in the first case, for any  $\delta \in (0,1)$, there exists $i \in \mathbb{N}$ such that $\omega (\delta, i) = \infty$. For example, let $\omega =\{q_n \}_{n=1}^{\infty}$ be a sequence which is monotonic decreasing and converges to zero as $n \to \infty$. Then, for any  $\delta \in (0,1)$, there exists $n_0 \in \mathbb{N}$ such that $q_n < \delta$ for any $n > n_0$, hence $\omega (\delta, i) = \infty$ for any $i > n_0$. In the second case, for some $\delta \in (0,1)$, $\omega (\delta, i) < \infty$ for any $i \in \mathbb{N}$. For example, let $\omega =\{q_n \}_{n=1}^{\infty}$ be the sequence defined by
\[
 q_n =
 \begin{cases}
 \frac{2}{3} & (n=2^m; m \in \mathbb{N})\\
 (\frac{1}{2})^n & (\text{otherwise}).
 \end{cases}
 \]
For any $i \in \mathbb{N}$, there exists $m \in \mathbb{N}$ such that $2^{m-1} \le i < 2^m$, hence for any $\displaystyle \delta \in (0, \frac{2}{3}]$ and $i \in \mathbb{N}$, $\omega(\delta ; i) =\inf \{ k \in \mathbb{N} \mid q_{i+k} \ge \delta \} \le 2^{m} - 2^{m-1} = 2^{m-1} < \infty$. (On the other hand,  if $\displaystyle \frac{2}{3} < \delta <1 $, $\omega(\delta ; i) =\infty$ for any  $i \in \mathbb{N}$, therefore   $N(\omega; \delta) = \infty$ for any $\delta \in (0,1)$.)

To solve our Question in general is very difficult, so we prove that it is true under some assumptions in the first case and in the  second case, respectively.
 
 \begin{theorem}\label{thm1} 
 Suppose that either a condition $\rm{(I)}$ or a condition $\rm{(II)}$ is satisified:
\begin{enumerate}
 \item[(I)] (The first case.) A sequence $\omega = \{ q_n\}_{n=1}^{\infty} \subset (0,1)$ is monotonic decreasing and converges to $0$ such that
 \[ q_n \log(\log (1/q_{n+1})) \to \infty\ (n \to \infty).\]
\item[(II)] (The second case.)
 For  $\omega = \{ q_n\}_{n=1}^{\infty} \subset (0,1)$, there exist sequences  $\{p_n\}_{n=1}^{\infty} \subset (0,1)$, $\mathcal{A}:=\{ a_m\}_{m=1}^{\infty} \subset \mathbb{N}$ and a constant $d \in (0,1)$ such that 
 
 \begin{enumerate}
 \item[(1)] $\{ p_n \}$ is monotonic decreasing, and converges to $0$ $(n \to \infty)$, 
 \item[(2)]  $0 < a_{m+1} - a_m \to \infty$ $(m \to \infty)$,
 \item[(3)] 
 \[ \lim_{m \to \infty} \sum_{n=a_m +1}^{a_{m+1}} \exp \left(\frac{-\pi^2}{2p_n}\right) = \infty\]
 and
 \item[(4)]   
 \[
 q_n =
 \begin{cases}
 d & (n \in \mathcal{A})\\
 p_n & (\text{otherwise}).
 \end{cases}
 \]
 \end{enumerate}
 \end{enumerate}
Then the two metrics $d_T$ and $d_L$ define different topologies on $T(X_{E(\omega)})$.
 \end{theorem}

 \begin{example}
 Let $\omega = \{ q_n\}_{n=1}^{\infty}$ be a sequence satisfying $q_{n+1}=1/\exp (\exp (n/ q_n))$, then $\omega  = \{ q_n\}_{n=1}^{\infty}$ satisfies the condition (I) of Theorem \ref{thm1}. Indeed, 
 \[ q_n \log(\log (1/q_{n+1})) = q_n \cdot (n/ q_n) =n \to \infty.\]
 The sequences satisfying the condition (I) is decreasing very rapidly.
 \end{example}
  
 \begin{example}
 For the sequences $\displaystyle \left\{ p_n = \frac{\pi^2}{2 \log n} \mid n \in \mathbb{N} \right\}$ and $\mathcal{A}=\{a_m \}_{m=1}^{\infty}$ satisfying $a_{m+1} = 2^m a_m$ and $a_1 =1$, define $\omega =\{ q_n\}_{n=1}^{\infty}$ as
 \[
  q_n =
 \begin{cases}
 \frac{1}{2} & (n \in \mathcal{A})\\
 p_n  & (\text{otherwise}).
 \end{cases}
 \]
 Then $\omega  = \{ q_n\}_{n=1}^{\infty}$ satisfies the condition (II) of Theorem \ref{thm1}. Indeed, 
 \[  \exp \left(\frac{-\pi^2}{2 \cdot p_n}\right) = \frac{1}{n},\]
 and
 \begin{eqnarray*}
 \sum_{n=a_m +1}^{a_{m+1}}  \exp \left(\frac{-\pi^2}{2 \cdot p_n}\right) 
 &=& \sum_{n=a_m +1}^{a_{m+1}} \frac{1}{n}\\
&=&  \sum_{n=a_m +1}^{2 a_m} \frac{1}{n} + \sum_{n=2a_m +1}^{2^2 a_m} \frac{1}{n}
  + \cdots + \sum_{n=2^{m-1} a_m +1}^{2^m a_m} \frac{1}{n}\\
 &>& a_m \cdot \frac{1}{2a_m} + 2a_m \cdot \frac{1}{2^2 a_m} + \cdots + 2^{m-1} a_m \cdot \frac{1}{2^m a_m}\\
 &=& \frac{1}{2} m \to \infty (m \to \infty).
  \end{eqnarray*}
  The sequences satisfying the condition (II) is decreasing slowly.
  \end{example}

 In this paper, we show lemmas to prove Theorem \ref{thm1} in Section 2, and prove Theorem \ref{thm1} in Section 3.
 
 \begin{remark}
In Section 2, we reveal that if $\omega = \{ q_n\}_{n=1}^{\infty}$ satisfies (I), then $X_{E(\omega)}$ is very similar to Shiga's example (\cite{Shiga1}, \S 3). However there are differences between them: 
\begin{enumerate}
\item He constructed the Riemann surface by gluing pairs of pants together. In his way, he can assign an arbitrary length to each geodesic of the boundary of each pair of pants. On the other hand, in this paper, we construct it by a generalized Cantor set $\omega  = \{ q_n\}_{n=1}^{\infty}$. In general, $q_n$ does not determine the length of each geodesic of the boundary of $n$-th pair of pants, which depends on not only $q_n$ but also $q_{\ell} \in \omega \setminus \{ q_n \}$. (However, if $\omega$ satisfies some conditions, then it depends only on $q_n$ (Remark \ref{explain}).) Lemmas  \ref{upper} and \ref{lower} in this paper are convenient propositions as the basis relating two kinds of Riemann surfaces. 

\item (Significant difference.) His Riemann surface does not agree with its convex core, that is, the corresponding Fuchsian group is of the second kind. To be a little more specific, in the case where a Riemann surface is constructed by gluing pairs of pants together, the corresponding Fuchsian group can be of the second kind if lengths of boundaries of pairs of pants increase very rapidly. Therefore he considered the two metrics on reduced Teichm\"uller spaces, and if someone would like to consider (not reduced) Teichm\"uller space for Shiga's Riemann surface, some device is needed. Matsuzaki  (\cite{Matsuzaki}) showed that Shiga's Riemann surface can be reconstructed into a Riemann surface corresponding to the Fuchsian group of the first kind by twist boundaries of pairs of pants when they are glued together. On the other hand, for our Riemann surface $X_{E(\omega)}$, no device is needed, because $X_{E(\omega)}$ is symmetric about $\mathbb{R} \cup \{ \infty \}$ (without twist) and corresponds to the Fuchsian group of the first kind for any $\omega = \{ q_n\}_{n=1}^{\infty}$.
\end{enumerate}
\end{remark}

 
\section{Lemmas to prove Theorem \ref{thm1}}

At the beginning, we decompose $X_{E(\omega)}$ by pants for an arbitrary $E(\omega)$: for any $k \in \mathbb{N}$ and $E_k= \bigcup_{i \in \mathcal{I}_k} I_k^i$, let $\{ \gamma_k^i \}_{i \in \mathcal{I}_k}$ be a family of disjoint simple closed curves in $\hat{\mathbb{C}}$ such that for each $i \in \mathcal{I}_k$, $\gamma_k^i$ separates $I_k^i$ and $\{I_k^{i'}\}_{i' \in \mathcal{I}_k \setminus \{i \}}$. (See Figure \ref{pb-x}.) Note that $\displaystyle \{ \gamma_k^i \mid i \in \mathcal{I}_k, k \in \mathbb{N}\}$ is regarded as a family of simple closed curves in $X_{E(\omega)}$. Also, $ \gamma_1^1$ and $ \gamma_1^2$ are homotopic, so we put $\gamma_1:=[\gamma_1^1] = [\gamma_1^2]$. 
 
 \begin{figure}[h]
\centering
\includegraphics[width=10cm,bb=0 0 681 664]{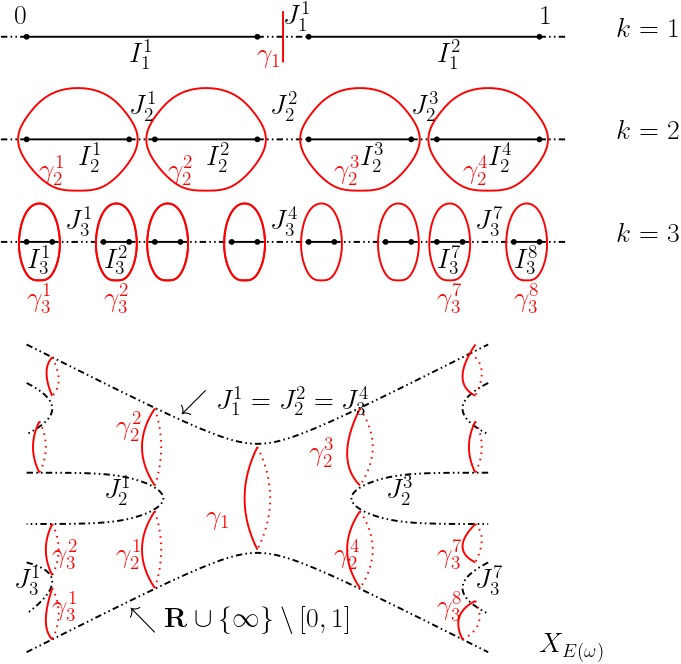}\\
\caption{$\displaystyle \{ \gamma_k^i \mid i \in \mathcal{I}_k\}$ $(k \le 3)$}
\label{pb-x}
\end{figure}   

 Let $P_1^1$ and $P_1^2$ be pairs of pants bounded by $\gamma_1$, $[\gamma_2^1]$, $[\gamma_2^2]$ and $\gamma_1$, $[\gamma_2^3]$, $[\gamma_2^4]$, respectively. And also, for any $k \ge 2$ and $i \in \mathcal{I}_k$, let $P_k^i$ be a pair of pants bounded by geodesics $[\gamma_k^i]$, $[\gamma_{k+1}^{2i -1}]$, $[\gamma_{k+1}^{2i}]$. Then $X_{E(\omega)}$ is decomposed by pants
 $\bigcup_{k=1}^{\infty} (\bigcup_{i \in \mathcal{I}_k} P_k^i)$.
 
Let us estimate lengths of geodesics $\displaystyle \{ [\gamma_k^i] \mid i \in \mathcal{I}_k, k \in \mathbb{N}\}$ in $X_{E(\omega)}$. To prove the following lemmas, we name the intervals: for each $k$ and each $j \in \mathcal{J}_k :=\{ 1,2,3,...,2^{k}-1 \}$, the $j$-th open interval from the left in $I \setminus E_k$ is denoted by $J_k^j$ (Figure \ref{pb-x}) and put $J_k^0 =J_k^{2^k}:= \mathbb{R} \cup \{ \infty\}\setminus I$. Then, for example, $J_1^1=J_2^2 = J_3^{2^2}=\cdots = J_k^{2^{k-1}} =\cdots$. In general, for any $k \in \mathbb{N}$ and any odd number $m \in \mathcal{J}_k$, $J_{k}^{m} = J_{k+1}^{2 m} = J_{k+2}^{2^2 m}=\cdots = J_{k + \ell}^{2^{\ell} m}=\cdots.$ 

Also, put

\[ U(x) :=\frac{2 \pi^2}{\log \frac{1+x}{1-x}} \left( =\frac{\pi^2}{\tanh^{-1} x} \right). \]

 \begin{lemma}  \label{upper}
 Let $\displaystyle \{ [\gamma_k^i] \mid i \in \mathcal{I}_k, k \in \mathbb{N}\}$ be closed geodesics in $X_{E(\omega)}$ as above.
 \begin{enumerate}
 \item If $i=1$ or $2^k$, then $\displaystyle \ell_{X_{E(\omega)}} ([\gamma_k^i]) < U(q_k)$ holds for any $k \in \mathbb{N}$. 
  
\item If $i \in \mathcal{I}_k \setminus \{ 1,2^k\}$, then
\[ \ell_{X_{E(\omega)}} ([\gamma_k^i]) < \max \left\{U(q_k) , \frac{2 \pi^2}{\log \frac{1-q_k + 2 q_{k-\ell}}{1-q_k}} \right\}, \]
where $\ell \in \{1,2,...,k-1 \}$ satisfies $i=2^{\ell} m$ or $i=2^{\ell}m +1$ for some odd number $m$. 
\end{enumerate}
    \end{lemma}
   
\begin{proof}
Note that for any $i \in \mathcal{I}_k$,
\begin{equation}\label{2-1}
|I_k^i| =\frac{1}{2} (1-q_k)|I_{k-1}^1|.
\end{equation}
Also, if $i \in \mathcal{J}_k$ is odd, then
\begin{equation}\label{2-2}
|J_k^i| = q_k |I_{k-1}^1|, 
\end{equation}
and if $i \in \mathcal{J}_k$ is even, then there exist $\ell \in \{1,2,...,k-1\}$ and an odd number $m$ such that $i=2^{\ell} m$, so 
\begin{equation}\label{2-3}
|J_k^i| = |J_k^{2^{\ell} m}| = |J_{k-\ell}^m| = q_{k- \ell} |I_{k-\ell -1}^1|.
\end{equation}

Now, let $i$ be an arbitrary number in $\mathcal{I}_k$. Firstly, we consider the case where $|J_k^{i-1}| > |J_k^i|$. For the midpoint $x_k^i$ of $I_k^i$ and a sufficiently small number $\varepsilon > 0$, take the annulus

\[ A_k^i := \{z \in \mathbb{C} \mid \frac{1}{2}|I_k^i|(1+\varepsilon)< |z-x_k^i| < \frac{1}{2}(|I_k^i| + |J_k^i|) (1+\varepsilon )\}.\] 

\begin{figure}[h]
\centering
\includegraphics[width=11cm,bb=0 0 799 312]{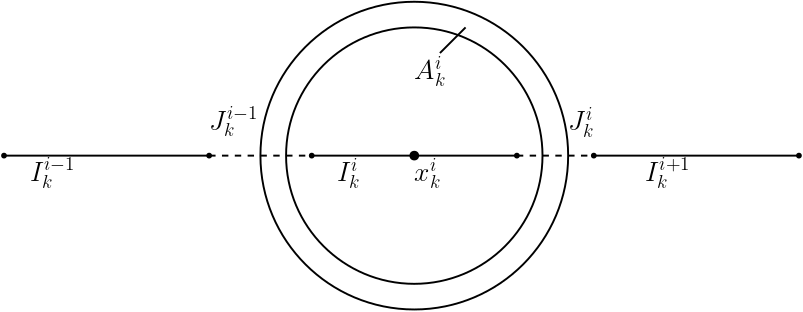}\\
\caption{Intervals $J_k^{i-1}$, $I_k^i$, $J_k^i$ and an annulus $A_k^i$.}
\label{annu-k-i}
\end{figure}

\textbf{The case where $i$ is odd.} By (\ref{2-1}) and (\ref{2-2}), the ratio $R_k^i$ of the radii of boundary circles of $A_k^i$ is
\begin{eqnarray*}
R_k^i=\frac{\frac{1}{2}|I_k^i|(1+\varepsilon) }{ \frac{1}{2}(|I_k^i| + |J_k^i|) (1+\varepsilon )} 
&=& \frac{|I_k^i|}{|I_k^i| + |J_k^i|} \\
&=& \frac{\frac{1}{2} (1-q_k)|I_{k-1}^1| }{\frac{1}{2}(1-q_k)|I_{k-1}^1| + q_k |I_{k-1}^1|} \\
&=&  \frac{\frac{1}{2} (1-q_k)}{\frac{1}{2}(1-q_k) + q_k} \\
&=& \frac{1-q_k}{1+q_k}.
\end{eqnarray*}
Hence, the length of the core curve $c_k^i$ in $A_k^i$ is $\displaystyle \frac{-2 \pi^2}{\log R_k^i}$ (cf. the proof of Theorem III in \cite{Shiga2}). Therefore 
\[ \ell_{X_{E(\omega)}} ([\gamma_k^i]) \le \ell_{X_{E(\omega)}} (c_k^i) \le \ell_{A_k^i} (c_k^i) = \frac{-2 \pi^2}{\log R_k^i} =  \frac{2 \pi^2}{\log (1/R_k^i)} = \frac{2 \pi^2}{\log \frac{1+q_k}{1-q_k}} . \]
In particular, if $i=1$, the inequality holds.

\textbf{The case where $i$ is even.} For $i$, let $\ell$ be the natural number satisfying $i=2^{\ell} m$ for some odd number $m$. Then $ |I_{k-\ell -1}^1| >  |I_{k -1}^1|$ holds by the definition of intervals $\{I_{k}^1 \}$. Hence, by (\ref{2-1}) and (\ref{2-3}), the ratio $R_k^i$ of the radii of boundary circles of $A_k^i$ is
\begin{eqnarray*}
R_k^i=\frac{\frac{1}{2}|I_k^i|(1+\varepsilon) }{ \frac{1}{2}(|I_k^i| + |J_k^i|) (1+\varepsilon )} 
&=& \frac{|I_k^i|}{|I_k^i| + |J_k^i|} \\
&=& \frac{ \frac{1}{2} (1-q_k)|I_{k-1}^1| }{\frac{1}{2}(1-q_k)|I_{k-1}^1| + q_{k-\ell} |I_{k-\ell -1}^1|} \\
&<& \frac{ \frac{1}{2} (1-q_k)|I_{k-1}^1| }{\frac{1}{2}(1-q_k)|I_{k-1}^1| + q_{k-\ell} |I_{k-1}^1|} \\
&=&  \frac{1-q_k}{1-q_k + 2 q_{k-\ell}}.
\end{eqnarray*}
Similarly as in the case where $i$ is odd,
\[ \ell_{X_{E(\omega)}} ([\gamma_k^i]) \le \frac{-2 \pi^2}{\log R_k^i} = \frac{2 \pi^2}{\log (1/R_k^i)} <  \frac{2 \pi^2}{\log ((1-q_k + 2 q_{k-\ell})/(1-q_k))}. \]

Secondly, if $|J_k^{i-1}| \le |J_k^i|$, take the annulus $A_k^i =\{z \in \mathbb{C} \mid \frac{1}{2}|I_k^i|(1+\varepsilon)< |z-x_k^i| < \frac{1}{2}(|I_k^i| + |J_k^{i-1}|) (1+\varepsilon )\}$ for the midpoint $x_k^i$ of $I_k^i$ and have a similar argument. If $i$ is odd, the ratio $R_k ^i < (1-q_k)/(1-q_k + 2 q_{k-\ell})$, where $\ell \in \{1,2,...,k-1 \}$ satisfies $i-1=2^{\ell} m$ for some odd number $m$. If $i$ is even, then the ratio $R_k^i  < (1-q_k)/(1+q_k)$. In particular, if $i= 2^k$, the inequality holds. 
\end{proof}   

\begin{remark}
By Lemma \ref{upper}, if $q_k \to 1$, then $\ell_{X_{E(\omega)}} ([\gamma_k^i]) \to 0$  as $k \to \infty$ $(i \in \mathcal{I}_k)$.
\end{remark}

\begin{remark} \label{explain}
To explain Lemma \ref{upper} more precisely, if $\omega =\{ q_n \}_{n=1}^{\infty}$ is monotonic decreasing, then $\ell_{X_{E(\omega)}}([\gamma_k^i]) <U(q_k)$ holds for any $k \in \mathbb{N}$ and $i \in \mathcal{I}_k$. Also, for an arbitrary $\omega =\{ q_n \}_{n=1}^{\infty}$, if $k$ satisfies that $q_k < q_{r}$ for any $r \in \{1,2,..., k-1\}$ (i.e. $\displaystyle \min_{1 \le r \le k} q_{r} =q_k$), then it holds for any $i \in \mathcal{I}_k$. Actually, from the proof, the inequality  
\begin{equation}\label{ineq}
\ell_{X_{E(\omega)}} ([\gamma_k^i]) \le  2\pi^2 / \log \left( (1-q_k + 2 q_{k-\ell})/ (1-q_k) \right)
\end{equation}
 holds if $k \in \mathbb{N}$ and $i \in \mathcal{I}_k \setminus \{1, 2^k\}$ satisfy
 \begin{equation}\label{ineq2}
 q_k > q_{k- \ell} \cdot 2^{\ell} \prod_{p=k - \ell}^{k-1} \frac{1}{1-q_p}, (q_k \ \text{is much larger than} \ q_{k-\ell},)
 \end{equation}
 where $\ell \in  \{1,2,...,k-1 \}$ satisfies $i=2^{\ell} m$ or $i=2^{\ell}m +1$ for some odd number $m$. Indeed, the inequality (\ref{ineq}) holds if either $|J_k^{i-1}| > |J_k^i|$ and $i$ is even or $|J_k^{i-1}| \le |J_k^i|$ and $i$ is odd. Now, $\displaystyle |I_k^i| =\frac{1}{2} (1-q_k)|I_{k-1}^1|=\left(\frac{1}{2}\right)^k \prod_{p=1}^k (1-q_p)$. Hence, by (\ref{2-2}) and (\ref{2-3}),  if $i=2^{\ell} m$, then 
\begin{eqnarray*}
\frac{|J_k^{i}|}{|J_k^{i-1}|}  
&=& \frac{q_{k- \ell} |I_{k-\ell-1}^1|}{q_{k} |I_{k-1}^1|} \\
&=& \frac{q_{k- \ell} \left(\frac{1}{2}\right)^{k-\ell-1} \prod_{p=1}^{k-\ell-1} (1-q_p)}{q_{k} \left(\frac{1}{2}\right)^{k-1} \prod_{p=1}^{k-1} (1-q_p)}\\
&=& \frac{q_{k- \ell}}{q_{k}} \cdot 2^{\ell} \cdot \prod_{p=k-\ell}^{k-1} \frac{1}{1-q_p},
\end{eqnarray*}
that is, $|J_k^{i-1}| > |J_k^i|$ means the inequality (\ref{ineq2}). Similarly, if $i=2^{\ell}m +1$, then $|J_k^{i-1}| \le |J_k^i|$ means the inequality (\ref{ineq2}). 

\end{remark}
Now, put
\[ L(x) := 2 \eta \left( \frac{2 \pi^2}{\log \frac{1+x}{2x}}\right), \]
where  $\eta(x)$ is the collar function: $\displaystyle \eta(x) = \sinh^{-1} \left(1/\sinh \frac{x}{2}\right)$.

\begin{lemma}\label{lower}
For each $k \in \mathbb{N}$ and each $i \in \mathcal{I}_k$, 
\[ \ell_{X_{E(\omega)}} ([\gamma_k^i]) > L(q_k) \]
holds. 
\end{lemma}

\begin{proof}
For any $k \in \mathbb{N}$ and $i \in \mathcal{I}_k$, the geodesic $[\gamma_k^i]$ in $X_{E(\omega)}$ is regarded as a curve in $\hat{\mathbb{C}}$, and it intersects open intervals $J_k^{i-1}$ and $J_k^{i}$. (See Figure \ref{b-k-i}.) Let $X_k^i$ be a four-punctured sphere defined by removing endpoints of $J_k^{i-1}$ and ones of $J_k^{i}$ from $\hat{\mathbb{C}}$. The geodesic $[\gamma_k^i]$ in $X_{E(\omega)}$ is regarded as a curve $\alpha_k^i$ in $X_k^i$ by the inclusion map $\iota: X_{E(\omega)} \hookrightarrow X_k^i$, hence $\ell_{X_k^i} (\alpha_k^i) \le \ell_{X_{E(\omega)}} ([\gamma_k^i])$ holds, so it is enough to show that

\begin{equation}\label{2.6}
\ell_{X_k^i} (\alpha_k^i)  > L(q_k).
\end{equation}

Firstly, we consider the case where $i$ is odd. For the midpoint $y_k^i$ of $J_k^i$ and a sufficiently small number $\varepsilon > 0$, take the annulus
\[ B_k^i := \{ z \in \mathbb{C} \mid \frac{1}{2}|J_k^i|(1+\varepsilon)< |z-y_k^i| < \frac{1}{2}(|J_k^i| + |I_k^i|) (1+\varepsilon )\}. \]

\begin{figure}[h]
\centering
\includegraphics[width=11cm,bb=0 0 793 299]{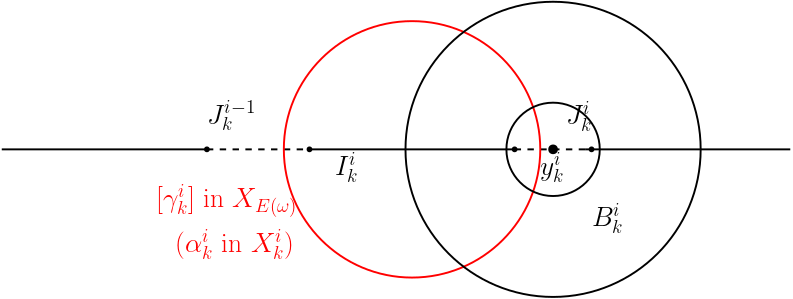}\\
\caption{$[\gamma_k^i]$ in $X_{E(\omega)}$ ($\alpha_k^i$ in $X_k^i$) and an annulus $B_k^i$.}
\label{b-k-i}
\end{figure}   

Then the ratio $S_k^i$ of the radii of boundary circles of $B_k^i$ is
\begin{eqnarray*}
S_k^i=\frac{\frac{1}{2}|J_k^i|(1+\varepsilon) }{ \frac{1}{2}(|J_k^i| + |I_k^i|) (1+\varepsilon )} 
&=& \frac{|J_k^i|}{|J_k^i| + |I_k^i|} \\
&=& \frac{q_k |I_{k-1}^1|}{q_k |I_{k-1}^1| + \frac{1}{2}(1-q_k)|I_{k-1}^1|} \\
&=& \frac{2q_k}{1+q_k}.
\end{eqnarray*}
Let $d_k^i$ be the core curve in $B_k^i$, then 
\[ \ell_{X_k^i} (d_k^i) \le  \ell_{B_k^i} (d_k^i) = \frac{-2 \pi^2}{\log S_k^i} =  \frac{2 \pi^2}{\log (1/S_k^i)} = \frac{2 \pi^2}{\log ((1+q_k)/2q_k)}.\]
Since the curve $d_k^i$ intersects $\alpha_k^i$ twice in $X_k^i$, we obtain the inequality (\ref{2.6}) by the collar lemma.

Secondly, suppose $i$ is even. For the midpoint $y_k^{i-1}$ of $J_k^{i-1}$ and a sufficiently small number $\varepsilon > 0$, take the annulus $B_k^{i-1} := \{ z \in \mathbb{C} \mid \frac{1}{2}|J_k^{i-1}|(1+\varepsilon)< |z-y_k^{i-1}| < \frac{1}{2}(|J_k^{i-1}| + |I_k^i|) (1+\varepsilon )\}$ and have a similar argument. Since $i-1$ is odd, the ratio $S_k^{i-1}$ of the radii of boundary circles of $B_k^{i-1}$ is $\displaystyle 2q_k/(1+q_k).$
\end{proof} 

\begin{remark}
By Lemma \ref{lower}, if $q_k \to 0$, $\ell_{X_{E(\omega)}} ([\gamma_k^i]) \to \infty$  $(k \to \infty)$ for any $i \in \mathcal{I}_k$.
\end{remark}

Next, we supplement the conditions of Theorem \ref{thm1}. In the following,  for functions $f(x),g(x)$, it is denoted $f(x) \sim g(x)$ ($x \to a)$ that $\displaystyle \lim_{x \to a} f(x)/g(x) =1$.

\begin{lemma}\label{con1}
Let $\omega = \{ q_n\}_{n=1}^{\infty}$ be a sequence satisfying the condition $\rm{(I)}$ of Theorem \ref{thm1}. Then, on $X_{E(\omega)}$, 
\[ \frac{\ell_{X_{E(\omega)}}([\gamma_{n+1}^{I}])}{\ell_{X_{E(\omega)}}([\gamma_n^i])} \to \infty \ (n \to \infty)\]
for any $i \in \mathcal{I}_n$ and any $I \in \mathcal{I}_{n+1}$.
\end{lemma}

\begin{proof}
By Lemmas \ref{upper}, \ref{lower} and Remark \ref{explain},
\begin{eqnarray*}
\frac{\ell_{X_{E(\omega)}}([\gamma_{n+1}^{I}])}{\ell_{X_{E(\omega)}}([\gamma_n^i])}
&>& \frac{L(q_{n+1})}{U(q_n)}\\
&=& \frac{ \sinh^{-1}(1/\sinh(2\pi^2 / \log \frac{1+q_{n+1}}{2q_{n+1}}))}{2\pi^2/\log \frac{1+q_n}{1-q_n}}\\
&=& \frac{1}{2\pi^2} \cdot \sinh^{-1}(1/\sinh(2\pi^2 / \log \frac{1+q_{n+1}}{2q_{n+1}})) \cdot \log \frac{1+q_n}{1-q_n}.
\end{eqnarray*}
Here, $\sinh^{-1} x \sim \log x$ $(x \to \infty)$ holds and $\log ((1+x)/2x) \sim \log (1/x)$  $(x \to 0)$, $\sinh x \sim x$ $(x \to 0)$ hold. Hence, if $q_n \to 0$,
\begin{eqnarray*}
\sinh^{-1}(1/\sinh(2\pi^2 / \log \frac{1+q_{n+1}}{2q_{n+1}}))
&\sim& \log(1/\sinh(2\pi^2 / \log \frac{1+q_{n+1}}{2q_{n+1}}))\\
&=& - \log(\sinh(2\pi^2 / \log \frac{1+q_{n+1}}{2q_{n+1}}))\\
&\sim& - \log(\sinh(2\pi^2 / \log \frac{1}{q_{n+1}}))\\
&\sim& - \log(2\pi^2 / \log \frac{1}{q_{n+1}})\\
&=& -\log 2\pi + \log(\log (1/q_{n+1}))\\
&\sim & \log(\log (1/q_{n+1})).
\end{eqnarray*}
Also, $\log ((1+x)/(1-x))  = 2x +2x^3/3 + \cdots \sim 2x$ $(x \to 0)$. Therefore
\[\frac{\ell_{X_{E(\omega)}}([\gamma_{n+1}^{I}])}{\ell_{X_{E(\omega)}}([\gamma_n^i])} > \frac{L(q_{n+1})}{U(q_n)} \sim \frac{1}{2\pi^2} \cdot \log(\log (1/q_{n+1})) \cdot 2q_n \to \infty \ (n \to \infty).\]
\end{proof}

From Lemma \ref{con1}, we obtain the following fact about pairs of pants $\{P_k^i \}$ (Figure \ref{p-k-i}) in cases where   sequences satisfy the condition $\rm{(I)}$. 

\begin{figure}[h]
\centering
\includegraphics[width=4.5cm,bb=0 0 363 393]{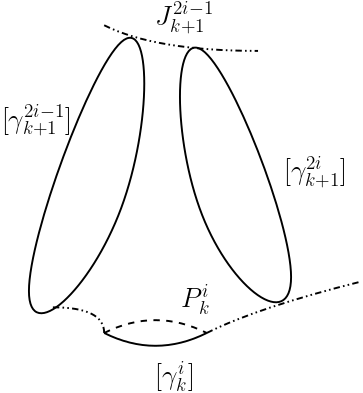}\\
\caption{A pair pf pants $P_k^i$.}
\label{p-k-i}
\end{figure}   

\begin{lemma}\label{pants1}
Let $\omega = \{ q_n\}_{n=1}^{\infty}$ be a sequence satisfying the condition $\rm{(I)}$ of Theorem \ref{thm1}. Then, for pairs of pants $\{P_k^i \}_{i \in \mathcal{I}_k, k \in \mathbb{N}}$, 
\[ \frac{d ([\gamma_k^i], J_{k+1}^{2i-1})}{{\ell_{X_{E(\omega)}}([\gamma_k^i])})} \to \infty \ (k \to \infty),\]
where $d(\cdot, \cdot)$ means the hyperbolic distance on $X_{E(\omega)}$.
\end{lemma}

\begin{proof}
For an arbitrary $k \in \mathbb{N}$ and $i \in \mathcal{I}_k$, decompose a pair of pants $P_k^i$ into two symmetric right-hexagons. And furthermore, decompose one of them into two right-pentagons, where a half of $[\gamma_k^i]$ (which is an edge of the right-hexagon) is decomposed into edges of two distinct right-pentagons. Let $p_{k+1}^{2i-1}$, $p_{k+1}^{2i}$ be pentagons with an edge $(1/2) [\gamma_{k+1}^{2i-1}]$ and with an edge $(1/2) [\gamma_{k+1}^{2i}]$, respectively. Now, let $a_{k,i}$, $a'_{k,i}$ be lengths of an edge $[\gamma_{k}^{i}] \cap p_{k+1}^{2i-1}$ and an edge $[\gamma_{k}^{i}] \cap p_{k+1}^{2i}$, respectively. We assume that $a_{k,i} \ge a'_{k,i}$, then $(1/4)\ell_{X_{E(\omega)}} [\gamma_{k}^{i}] \le a_{k,i} \le (1/2)\ell_{X_{E(\omega)}} [\gamma_{k}^{i}]$. Also, put $b_{k,i} := d ([\gamma_k^i], J_{k+1}^{2i-1})$ and $d_{k,i} := (1/2) \ell_{X_{E(\omega)}} [\gamma_{k+1}^{2i-1}]$. Then, by the formula of right-pentagons (cf. \cite{Beardon}),
\[ \cosh d_{k,i} = \sinh a_{k,i} \sinh b_{k,i}.\]
Here, $\cosh x \sim e^x /2$ and $\sinh x \sim e^x /2$ $(x \to \infty)$, thus
\[2\exp (d_{k,i}) \sim  \exp (a_{k,i} + b_{k,i}) \ (k \to \infty). \]
Since $d_{k,i} / a_{k,i} \to \infty$ $(k \to \infty)$ by Lemma \ref{con1},
\[ 2^{1/a_{k,i}} \exp (d_{k,i} / a_{k,i}) \sim  \exp (1 + b_{k,i} /a_{k,i})  (k \to \infty). \]
Therefore
\begin{eqnarray*}
\exp\left( 1+ \frac{d ([\gamma_k^i], J_{k+1}^{2i-1})}{{(1/4)\ell_{X_{E(\omega)}}([\gamma_k^i])})} \right) 
&\ge& \exp \left( 1+ \frac{b_{k,i}}{a_{k,i}} \right) \sim \exp \left( \frac{d_{k,i}}{a_{k,i}}\right)\\
&\ge& \exp \left( \frac{\ell_{X_{E(\omega)}}([\gamma_{k+1}^{2i-1}])}{\ell_{X_{E(\omega)}}([\gamma_k^i])}\right).
\end{eqnarray*}
By Lemma \ref{con1}, the statement is followed. It is similar in the case where $a_{k,i} \le a'_{k,i}$. 
\end{proof}

Next, we consider the condition (II) of Theorem \ref{thm1}.
\begin{lemma}\label{anno}
Let $\{p_n\}_{n=1}^{\infty} \subset (0,1)$ and $\{ a_m\}_{m=1}^{\infty} \subset \mathbb{N}$ be sequences satisfying the conditions $\rm{(1)},\rm{(2)},\rm{(3)}$ of $\rm{(II)}$ of Theorem \ref{thm1}. Then
\[ \lim_{m \to \infty} \sum_{n=a_m +1}^{a_{m+1}} \eta (U(p_n)) = \infty\]
holds, where $\eta$ is the collar function.
\end{lemma}

\begin{proof}
By the definitions of functions $\eta$ and $U$,
\begin{eqnarray*}
\eta (U(x)) = \sinh^{-1} \left( \frac{1}{\sinh \frac{\pi^2}{\log ((1+ x)/(1-x))}}\right).
\end{eqnarray*}
Also, $\displaystyle \sinh^{-1} x \sim x$ $(x \to 0)$,  $\displaystyle \sinh x \sim \frac{1}{2}\exp x$ $(x \to \infty)$ and $\displaystyle \log \frac{1+x}{1-x} = 2x +\frac{2x^3}{3} + \cdots$, hence 
\[ \eta(U(x)) \sim \frac{1}{\sinh \frac{\pi^2}{\log ((1+ x)/(1-x))}} \sim \frac{2}{\exp \frac{\pi^2}{\log ((1+ x)/(1-x))}}  \sim \frac{2}{\exp \frac{\pi^2}{2x}}  =2\exp \left( \frac{-\pi^2}{2x} \right) \]
$(x \to 0)$.
\end{proof}

Finally, we use the following to prove Theorem \ref{thm1}.

\begin{theorem}[K. (\cite{Kinjo1})] \label{keythm}
Suppose that for a hyperbolic Riemann surface $X$, there exists a family $\{ \alpha_n\}_{n=1}^{\infty} \subset \mathscr{C} (X)$ of simple closed geodesics such that for any geodesics $\{ \beta_n\}_{n=1}^{\infty} \subset \mathscr{C} (X)$ with $\alpha_n \cap \beta_n \neq \emptyset$ ($n=1,2,...$), 
\[ \lim_{n \to \infty} \frac{\ell_{X} (\beta_n)}{\sharp (\alpha_n \cap \beta_n)\ell_{X} (\alpha_n)} = \infty\]
holds. Then metrics $d_T$ and $d_L$ define different topologies.
\end{theorem}

This theorem means the following: suppose that for a closed geodesic $\alpha$ in $X$, any closed geodesic $\beta$ crossing $\alpha$ is much longer than $\alpha$. Then a Dehn twist $f$ along $\alpha$ almost never changes lengths of any closed geodesics in $X$, but it changes conformal structure near $\alpha$, that is, the length spectrum distance $d_L([X,id],[X,f])$ is almost zero, but the Teichm\"uller distance $d_T([X,id],[X,f])$ is away from zero.


\section{Proof of Theorem \ref{thm1}}
\subsection{Condition (I)}
Let $\omega = \{ q_n\}_{n=1}^{\infty}$ be a sequence satisfying the condition $\rm{(I)}$ of Theorem \ref{thm1}. Consider a family $\{ [\gamma_n^i] \mid i=1, n \in \mathbb{N}\}$ of simple closed geodesics taken in Section 2. Then the geodesics $\{ [\gamma_{n}^1] \}_{n=1}^{\infty}$ satisfies the condition of Theorem \ref{keythm}. Indeed, for any $n \in \mathbb{N}$, let $\beta_n$ be an arbitrary simple closed geodesic crossing $[\gamma_n^1]$. Then there exists $k \in \mathbb{N}$ such taht $\beta_n \cap J_{n + k}^{i} \neq \emptyset$ for some odd number $i$. If $\sharp ([\gamma_n^1] \cap \beta_n) =N \in \mathbb{N}$, then there exist $M \ge (1/2)N$ such numbers $k_1,...,k_{M}$. Put $\displaystyle k := \min \{k_1,...,k_{M}\}$. Since $\beta_n$ cross $[\gamma_{n+k-1}^{*}]$ and $J_{n + k}^{i} $, $\ell_{X(E(\omega))} (\beta_n) > M d([\gamma_{n+k-1}^{*}], J_{n + k}^{i} )$. Also $\ell_{X(E(\omega))} ([\gamma_n^1]) \le \ell_{X(E(\omega))} ([\gamma_{n+k-1}^1])$. Therefore
\[ \frac{\ell_{X(E(\omega))} (\beta_n)}{\sharp ([\gamma_n^1] \cap \beta_n)\ell_{X(E(\omega))} ([\gamma_n^1])} \ge  \frac{(1/2)N d([\gamma_{n+k-1}^{*}], J_{n + k}^{i} )}{N \ell_{X(E(\omega))} (\gamma_{n+k-1}^1)} \to \infty \ (n \to \infty)\]
by Lemma \ref{pants1}.
\subsection{Condition (II)}
Let $\omega = \{ q_n\}_{n=1}^{\infty} \subset (0,1)$ be a sequence with sequences  $\{p_n\}_{n=1}^{\infty} \subset (0,1)$, $\mathcal{A}=\{ a_m\}_{m=1}^{\infty} \subset \mathbb{N}$ and a constant $d \in (0,1)$ satisfying the condition (II) of Theorem \ref{thm1}. From geodesics $\{ [\gamma_n^i]\}_{i \in \mathcal{I}_k, k \in \mathbb{N}}$ taken in Section 2, choose a family of simple closed geodesics such that $n=a_m$ ($m=1,2,...$) and $i=1$. It is enough to show that the geodesics $\{ [\gamma_{a_m}^1] \}_{m=1}^{\infty}$ satisfies the condition of Theorem \ref{keythm}. To be more specific, let $\beta_m$ be an arbitrary simple closed geodesic crossing $[\gamma_{a_m}^1]$ $(m=1,2,...)$ and we shall show that $\ell_{X_{E(\omega)}} (\beta_m)\to \infty$ as $m \to \infty$. Note that $\ell_{X_{E(\omega)}} ([\gamma_{a_m}^1]) \le U(d)$ by Lemma \ref{upper}, and if $n \notin \mathcal{A}$, then $\ell_{X_{E(\omega)}} ([\gamma_{n}^i]) \to \infty$ as $q_n \to 0$ for any $i \in \mathcal{I}_n$ by Lemma \ref{lower}. Also, note that  $X_{E(\omega)}$ and each geodesic of $\{ [\gamma_n^i]\}_{i \in \mathcal{I}_k, k \in \mathbb{N}}$ are symmetric about $\mathbb{R} \cup \{ \infty\}$ by the definitions.

\begin{figure}[h]
\centering
\includegraphics[width=12.2cm,bb=0 0 883 414]{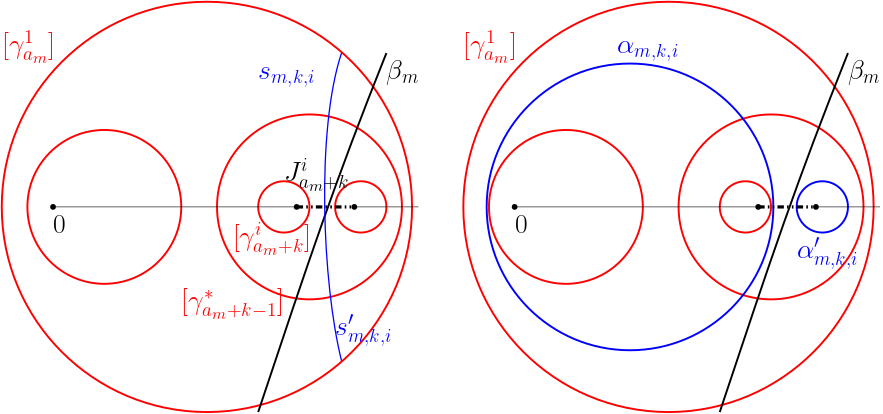}\\
\caption{$[\gamma_{a_m}^1$],  $\beta_m$, $J_{a_m +k}^{i}$ and $\alpha_{m,k,i}$, etc.}
\label{s-alpha}
\end{figure}

Now, for any $\beta_m$, there exist $k \in \mathbb{N}$ such $\beta_m \cap J_{a_m + k}^{i} \neq \emptyset$ for some odd number $i$. If $k$ satisfies that $a_m +k \in \mathcal{A}$, then $a_m + k \ge a_{m+1}$, so $\beta_m$ crosses closed geodesics $\{ [\gamma_n^{\ast}] \mid a_m < n < a_{m+1}\}$. If $m$ is sufficiently large, then each $q_n =p_n$ $(a_m < n < a_{m+1})$ is smaller than $q_r$ for any $r \in \{1,2,...,n-1\}$, hence $\ell_{X(E(\omega))} ([\gamma_n^i]) < U(q_n)$ for any $i \in \mathcal{I}_n$ (by Remark \ref{explain}). Therefore $\displaystyle \ell_{X_{E(\omega)}} (\beta_m) >  \sum_{n=a_m}^{a_{m+1}-1} \eta (U(q_n)) \to \infty$ $(m \to \infty)$ by Lemma \ref{anno}.

 In the following, suppose that $k$ satisfies that $a_m + k \notin \mathcal{A}$. Let $s_{m, k, i}$ be the shortest geodesic segment from $[\gamma_{a_m}^1]$ to $J_{a_m+k}^{i}$ and $s'_{m,k,i}$ be the geodesic segment given by reflecting $s_{m,k,i}$ across $\mathbb{R} \cup \{ \infty\}$. Then the connected segment $S_{m,k,i} :=s_{m,k,i} \cdot s'_{m,k,i}$ divides $[\gamma_{a_m}^1]$ into two geodesic segments. Regard $[\gamma_{a_m}^1] \cup S_{m,k,i}$ as two closed curves (with the intersection $S_{m,k,i}$) and take the two simple closed geodesics $\alpha_{m,k,i}, \alpha'_{m,k,i}$ which are homotopic to them respectively, where $\alpha_{m,k,i} \cap (\mathbb{R} \cup \{\infty \} \setminus [0,1])  \neq \emptyset$. (See the right of Figure \ref{s-alpha}.) 

\begin{clm} \label{m,k,i}
\[ \ell_{X(E(\omega))} (\alpha_{m,k,i}) > 2 \eta \left( \frac{2 \pi^2}{\log \frac{1+q_{a_m + k }}{2q_{a_m + k}}}\right),\ \]
where  $\eta(x)$ is the collar function. In particular, as $m \to \infty$, $\ell_{X(E(\omega))} (\alpha_{m,k,i}) \to \infty$. 
\end{clm}

\begin{proof}
Similarly as in the proof of Lemma \ref{lower}, let $X_{m,k,i}$ be a four-punctured Riemann surface given by removing two endpoints of $J_{a_m + k }^{i}$ and  $0,1$ from $\hat{\mathbb{C}}$. Let $y_{m,k,i}$ be the midpoint of $J_{a_m + k}^{i}$, and for a sufficiently small number $\varepsilon > 0$, take the annulus
\[ B_{m,k,i} := \{ z \in \mathbb{C} \mid \frac{1}{2}|J_{a_m + k}^{i}|(1+\varepsilon)< |z-y_{m,k,i}| < \frac{1}{2}(|J_{a_m + k }^{i}| + |I_{a_m + k}^{i}|) (1+\varepsilon )\}. \]
Then the ratio $S_{m,k,i}$ of the radii of boundary circles of $B_{m,k,i}$ is
\begin{eqnarray*}
S_{m,k,i}&=&\frac{\frac{1}{2}|J_{a_m + k}^{i}|(1+\varepsilon) }{ \frac{1}{2}(|J_{a_m + k}^{i}| + |I_{a_m + k}^{i}|) (1+\varepsilon )} \\
&=& \frac{q_{a_m +k} |I_{a_m + k-1}^{1}|}{q_{a_m +k} |I_{a_m + k-1}^{1}| + \frac{1}{2}(1-q_{a_m +k})|I_{a_m + k-1}^{1}|} \\
&=& \frac{2q_{a_m +k}}{1+q_{a_m +k}}.
\end{eqnarray*}

Therefore the core curve $\delta_{m,k,i}$ in $B_{m,k,i}$ satisfies 
\[ \ell_{X_{m,k,i}} (\delta_{m,k,i}) \le  \ell_{B_{m,k,i}} (\delta_{m,k,i}) =  \frac{2 \pi^2}{\log (1/S_{m,k,i})} = \frac{2 \pi^2}{\log ((1+q_{a_m +k})/2q_{a_m +k})}.\]

By the collar lemma, $ \ell_{X_{m,k,i}} (\alpha_{m,k,i}) > 2 \eta ( \ell_{X_{m,k,i}} (d_{m,k,i}) )$ holds, and by Schwarz lemma, the desired inequality is verified.
\end{proof}

Consider a pair of pants bounded by $\alpha_{m,k,i}, \alpha'_{m,k,i}$ and $[\gamma_{a_m}^1]$ and divide it into two symmetric right-hexagons. Note that the pants is symmetric about $\mathbb{R} \cup \{ \infty\}$ by the definition, so the dividing geodesic segments are included in $\mathbb{R} \cup \{ \infty\}$, in particular, the segment $\sigma_{m,k,i}$ connecting $\alpha_{m,k,i}$ and $\alpha'_{m,k,i}$ is included in $J_{a_m + k}^{i}$. Divide one of right-hexagons into two right-pentagons and put $a_{m,k,i} := d([\gamma_{a_m}^1], \alpha_{m,k,i})$, $b_{m,k,i} :=(1/2)\ell_{X_{E(\omega)}} (\alpha_{m,k,i})$  and $d_{m,k,i} := d(\sigma_{m,k,i}, [\gamma_{a_m}^1])$. Then, by the formula of right-pentagons (cf. \cite{Beardon}),

\[ \cosh (d_{m,k,i}) = \sinh (a_{m,k,i}) \sinh (b_{m,k,i}).\]

By Lemma \ref{upper} and the collar lemma, $a_{m,k,i} > \eta(U(d)) >0$, and by Claim \ref{m,k,i}, $b_{m,k,i} \to \infty$ $(m \to \infty)$, therefore $d_{m,k,i} \to \infty$ $(m \to \infty)$, that is, $\ell_{X_{E(\omega)}} (\beta_m)\to \infty$.

\bibliographystyle{amsplain}

\end{document}